\documentclass[11pt,oneside,english]{amsart}
\usepackage{lmodern}

\usepackage[T1]{fontenc}
\usepackage[latin9]{inputenc}
\usepackage{geometry}
\usepackage{amsthm}
\usepackage{amstext}
\usepackage{amssymb}
\usepackage{setspace}
\makeatletter
\numberwithin{equation}{section}
\numberwithin{figure}{section}
\theoremstyle{plain}
\newtheorem{thm}{\protect\theoremname}[section]
  \theoremstyle{plain}
  \newtheorem{conjecture}[thm]{\protect\conjecturename}
  \theoremstyle{definition}
  \newtheorem{defn}[thm]{\protect\definitionname}
  \theoremstyle{remark}
  \newtheorem{rem}[thm]{\protect\remarkname}
  \theoremstyle{plain}
  \newtheorem{lem}[thm]{\protect\lemmaname}

\makeatother

\usepackage{babel}
  \providecommand{\conjecturename}{Conjecture}
  \providecommand{\definitionname}{Definition}
  \providecommand{\lemmaname}{Lemma}
  \providecommand{\remarkname}{Remark}
\providecommand{\theoremname}{Theorem}

\begin{document}

\title{Semicharacters of Groups}


\author{Gil Alon}
\begin{abstract}
We define the notion of a semicharacter of a group $G$: A function
from the group to $\mathbb{C}^{*}$, whose restriction to any abelian
subgroup is a homomorphism. We conjecture that for any finite group,
the order of the group of semicharacters is divisible by the order
of the group. We prove that the conjecture holds for some important
families of groups, including the Symmetric groups and the groups
$GL(2,q)$.
\end{abstract}
\maketitle

\section{Introduction}

If $G$ is a finite abelian group, the dual group $\widehat{G}=\hom(G,\mathbb{C}^{*})$
is isomorphic to $G$. For a general finite group $G$, $\hom(G,\mathbb{C}^{*})$
may not provide much information about the group: For example, when
$G$ is simple and nonabelian, $\hom(G,\mathbb{C}^{*})=\{1\}.$ In
this paper we look not only at homomorphisms from $G$ to $\mathbb{C}^{*}$,
but at a wider class of maps: Let us call a function $f:G\rightarrow\mathbb{C}^{*}$
a \emph{linear semicharacter} (or, more briefly, a \emph{semicharacter)}
if it satisfies $f(ab)=f(a)f(b)$ for all pairs $a,b\in G$ of \emph{commuting
}elements. In other words, a semicharacter on $G$ is a function whose
restriction to any abelian subgroup is a character.

The set of semicharacters of $G$, with the operation of pointwise
multiplication, is an abelian group. Let us denote this group by $\widehat{G}$,
as this notation is consistent with the case where $G$ is abelian. 

It is easy to see (lemma \ref{lem:Finiteness}) that $\widehat{G}$
is always finite. The first question that one may ask is whether $\widehat{G}$
is always nontrivial whenever $G$ is. This is true (see Theorem \ref{thm:For-any-finite}),
and the proof that we have uses the fact that any simple group has
a cyclic Sylow subgroup. This fact depends on the classification of
finite simple groups.

More generally, we will be interested in the relation between the
order of $\widehat{G}$ and the order of $G.$ We have investigated
this relationship for several central families of groups, listed below,
and also for all groups of order $\leq255$ using a computer (see
section \ref{sec:Numerical-evidence}). In all these cases, we have
noticed a remarkable phenomenon: The order of $\widehat{G}$ is divisible
by the order of $G$. This leads us to following conjecture:
\begin{conjecture}
\label{conjecture}For any finite group $G$, $|\widehat{G}|$ is
divisible by $|G|$.
\end{conjecture}
The conjecture obviously holds for abelian groups, in which case $\widehat{G}\cong G$.
We will prove that the conjecture holds for the following families
of groups:
\begin{itemize}
\item The Symmetric groups (Theorem \ref{thm:Sn}).
\item The Alternating groups (Theorem \ref{thm:An}).
\item The Linear groups $GL(2,q)$ (Theorem \ref{thm:GL(n)}).
\item The Unitriangular groups $U(n,q)$ when $q=p^{e}$ and $p>n$ (Theorem
\ref{thm:unipotent}).
\end{itemize}
We will prove all these statements by constructing semicharacters
explicitly. The techniques for construction depend heavily on the
structure of the given group, and some of them seem to be of independent
interest.

As an example, let us look at the Heisenberg group over a finite field
$\mathbb{F}$. Consider the three functions from the group to the
additive group $\mathbb{F}^{+}$, defined by 
\[
f_{1}\left(\left(\begin{array}{ccc}
1 & a & b\\
0 & 1 & c\\
0 & 0 & 1
\end{array}\right)\right)=a;\, f_{2}\left(\left(\begin{array}{ccc}
1 & a & b\\
0 & 1 & c\\
0 & 0 & 1
\end{array}\right)\right)=c;\, f_{3}\left(\left(\begin{array}{ccc}
1 & a & b\\
0 & 1 & c\\
0 & 0 & 1
\end{array}\right)\right)=ac-2b
\]
$f_{1}$ and $f_{2}$ are group homomorphisms. $f_{3}$ is not a homomorphism,
but it satisfies $f(AB)=f(A)+f(B)$ whenever $A$ and $B$ commute,
as can be verified directly. By taking linear combinations of these
$3$ functions, and composing them with a homomorphism from $\mathbb{F}^{+}$
to $\mathbb{C}^{*}$, we get enough semicharacters on the group.

In the proof of Theorem \ref{thm:unipotent}, we will generalize this
construction to upper triangular unipotent matrices via a discrete
version of the $\log$ function.

For groups $G$ that are not $p$-groups, we consider the $l$-part
of $\widehat{G}$ separately for each prime divisor $l$ of $|G|$,
and prove a localization principle: the $l$-part of $\widehat{G}$
depends only on the $l$-Sylow subgroups of $G$ and their intersection
pattern (see lemma \ref{lem:(1)-A-function}). For the symmetric group,
we will construct semicharacters using the cycle decomposition of
permutations. Basically, one is free to define the value of a semicharacter
in the $l$-torsion part of $\widehat{S_{n}}$ on cycles of maximal
$l$-power size. For the linear groups $GL(2,p^{t})$, we will use
various techniques to construct semicharacters, depending on $p^{t}$
and $l$. At $l=p$ our construction is similar to the case of the
unitriangular groups.

Although we currently do not have an application to the concept of
the semicharacter group, we find this construction interesting in
its own right. Such a simple definition brings about what seems to
be a natural and nontrivial conjecture, and the techniques for proving
the conjecture for various families of groups are quite diverse. Aside
from conjecture \ref{conjecture}, other interesting questions are
whether some information about $\widehat{G}$ can be inferred from
the groups $\widehat{G_{i}}$, if $G_{1},..,G_{k}$ are the composition
factors of $G$; whether conjecture \ref{conjecture} can be proved
for all simple groups; and whether theorem \ref{thm:For-any-finite}
can be proved independently of the classification of finite simple
groups.

\section{Some notation and lemmas}

All the groups in this paper are assumed to be finite. We repeat the
definitions in the introduction: 
\begin{defn}
(1) A \emph{linear semicharacter }on a group $G$ is a function $f:G\rightarrow\mathbb{C}^{*}$
satisfying $f(ab)=f(a)f(b)$ for all pairs $a,b\in G$ of\emph{ }commuting\emph{
}elements.

(2) The group of linear semicharacters of $G$ is denoted by $\widehat{G}$.\end{defn}
\begin{rem}
For the sake of brevity, we will use the term \emph{semicharacters
}for linear semicharacters. We must, however, caution the reader that
the term \emph{semicharacters} has other meanings in the contexts
of Lie groups and semigroups.
\end{rem}
We obtain a contravariant functor~ $\widehat{}$ ~from groups to
abelian groups, sending a group $G$ to $\widehat{G}$, and a group
homomorphism $\phi:G\rightarrow H$ to the homomorphism $\widehat{\phi}:\widehat{H}\rightarrow\widehat{G}$
defined by $\widehat{\phi}(f)=f\circ\phi$.
\begin{lem}
\label{lem:Finiteness}For any $G$, $\widehat{G}$ is a finite abelian
group, and its order divides $\prod_{g\in G}ord(g)$.\end{lem}
\begin{proof}
If $f\in\widehat{G}$ then for every $g\in G$, $f(g)^{ord(g)}=f(g^{ord(g)})=1$.
Hence, $\widehat{G}$ embeds in the group of functions $f:G\rightarrow\mathbb{C}^{*}$
satisfying $f(g)^{ord(g)}=1$ for all $g\in G$. Since the latter
group is finite of order $\prod_{g\in G}ord(g)$, the result follows
immediately.\end{proof}
\begin{defn}
For any group $G$, and any prime number $l$, we denote by $G[l^{\infty}]$
the set of elements whose order is a power of $l$, and by $G[l]$
the set $\{g\in G|g^{l}=id\}$.
\end{defn}
If $G$ is abelian, then $G[l^{\infty}]$ and $G[l]$ are subgroups
of $G$. We have

\[
\widehat{G}=\prod_{l}\widehat{G}[l^{\infty}]
\]

And each $\widehat{G}[l^{\infty}]$ is of $l$-power size. We will
now see that $\widehat{G}[l^{\infty}]$ depends only on the elements
of $G[l^{\infty}]$ and their multiplication table. Let us first introduce
a convenient notation:
\begin{defn}
For any nonempty subset $A\subseteq G$, we denote by $\widehat{A}$
the group of maps $f:A\rightarrow\mathbb{C}^{*}$ such that $f(xy)=f(x)f(y)$
whenever $x,y$ are commuting elements in $A.$ 
\end{defn}
Note that a function $f:G[l^{\infty}]\rightarrow\mathbb{C}^{*}$ is
in $\widehat{G[l^{\infty}]}$ if and only if for any $l$-Sylow subgroup
$L\leq G$, $f|_{L}\in\widehat{L}$. 
\begin{lem}
\label{lem:(1)-A-function}The restriction map from $\widehat{G}[l^{\infty}]$
to $\widehat{G[l^{\infty}]}$ is an isomorphism. \end{lem}
\begin{proof}
We have to prove that each function in $\widehat{G[l^{\infty}]}$
has a unique extension to $G$ which is in $\widehat{G}[l^{\infty}].$
Let $|G|=nl^{a}$ with $l\nmid n$, and let $b\in\mathbb{Z}$ be such
that $bn\equiv1(mod$ $l^{a})$. Then $\tilde{f}(g)=f(g^{n})^{b}$
is the required extension of $f$. The extension is unique because
any $f\in\widehat{G}[l^{\infty}]$ satisfies $f(g)=f(g^{n})^{b}.$
\end{proof}
Lemma \ref{lem:(1)-A-function} will be our main tool for constructing
semicharacters on groups. 
\begin{rem}
It follows from lemma \ref{lem:(1)-A-function} that $\widehat{G}[l]\cong\widehat{G[l^{\infty}]}[l]$.
We will view elements of the right hand side as functions $f:G[l^{\infty}]\rightarrow\mathbb{F}_{l}^{+}$
satisfying $f(ab)=f(a)+f(b)$ if $ab=ba$. It follows that if we can
find $d$ such functions which are linearly independent over $\mathbb{F}_{l}$,
then $|\widehat{G}|$ is divisible by $l^{d}$.
\end{rem}

\section{Nontriviality}
\begin{lem}
\label{lem:CyclicSylow}Let $G$ be a finite group, and $l|\,|G|$
a prime such that the $l$-Sylow subgroups of $G$ are cyclic. Let
$N$ be the number of $l$-Sylow subgroups of $G$. Then $|\widehat{G}|$
is divisible by $l^{N}.$\end{lem}
\begin{proof}
Let $L_{1},\dotsc,L_{N}$ be the $l$-Sylow subgroups of $G$. For
any choice of homomorphisms $\phi_{i}:L_{i}\rightarrow\mathbb{F}_{l}^{+}$
for $1\leq i\leq N$, and any $1\leq i<j\leq N$, $\phi_{i}$ and
$\phi_{j}$ are both trivial on $L_{i}\cap L_{j}$ (since $L_{i}\cap L_{j}$
does not contain generators of $L_{i}$ or $L_{j}$, hence each element
$L_{i}\cap L_{j}$ is an $l$-power of elements of both $L_{i}$ and
$L_{j}$). Hence the $N$-tuple $(\phi_{1},\dotsc,\phi_{N})$ glues
to a function $\phi:G[l^{\infty}]\rightarrow\mathbb{F}_{l}^{+}$.
If $g,h\in G[l^{\infty}]$ commute then they lie in a common $l$-Sylow
subgroup, hence $\phi(gh)=\phi(g)+\phi(h)$. By lemma \ref{lem:(1)-A-function}
we get a monomorphism from $\prod_{i}\hom(L_{i},\mathbb{F}_{l}^{+})$
(which is of size $l^{N}$) to $\widehat{G}$.\end{proof}
\begin{thm}
For any finite group $G\neq\{1\}$, $\widehat{G}\neq\{1\}$.\label{thm:For-any-finite}\end{thm}
\begin{proof}
Since a semicharacter on a quotient of $G$ can be pulled back to
$G$, we may assume that $G$ is simple. By \cite[Theorem 4.9]{Kimmerle},
$G$ has a cyclic Sylow subgroup. The claim now follows from lemma
\ref{lem:CyclicSylow}.
\end{proof}

\section{The symmetric and alternating groups}

For $G=S_{n}$, the primes dividing $|G|$ are exactly the primes
$l\leq n$. Let $l$ be such a prime, and let $e$ be such that $l^{e}\leq n<l^{e+1}$.
\begin{lem}
\label{lem:If--commute,}If $\alpha,\beta\in S_{n}[l^{\infty}]$ commute,
and if $\alpha_{1}$ (resp. $\beta_{1}$) is an $l^{e}$- cycle in
the cycle decomposition of $\alpha$ (resp. $\beta$), then either
$\alpha_{1}$ and $\beta_{1}$ are disjoint, or each one is a power
of the other.\end{lem}
\begin{proof}
Since $\beta^{\alpha}=\beta$, $\alpha$ acts on the $l^{e}$-cycles
of $\beta$. Since $\alpha$ is of $l$-power order, all the orbits
of this action are of $l$-power size, but $l\cdot l^{e}>n$, hence
$\alpha$ fixes each $l^{e}$-cycle of $\beta$. In particular, $\alpha_{1}$
and $\beta_{1}$ commute. The result follows. \end{proof}
\begin{thm}
\label{thm:Sn}For all $n$, conjecture \ref{conjecture} holds for
$S_{n}$. \end{thm}
\begin{proof}
For any prime $l\leq n$, let $A$ be the set of $l^{e}$-cycles of
$S_{n}$, and define $a\sim b$ if $b$ is a power of $a$. This is
an equivalence relation, and we have $|A/\sim|=\left(\begin{array}{c}
n\\
l^{e}
\end{array}\right)\frac{(l^{e}-1)!}{l^{e}(1-\frac{1}{l})}$. Let $V$ be the $\mathbb{F}_{l}$-vector space of functions $f:A\rightarrow\mathbb{F}_{l}^{+}$
satisfying $f(\pi^{i})=i\cdot f(\pi)$ if $(i,l)=1$. Then $dim(V)=|A/\sim|.$
For any $f\in V$, we extend $f$ to a function $\tilde{f}:S_{n}[l^{\infty}]\rightarrow\mathbb{F}_{l}^{+}$
by $\tilde{f}(\pi)=\sum_{c\in c(\pi,l^{e})}f(c)$, where $c(\pi,k)$
is the set of $k$-cycles in the cycle decomposition of $\pi$. By
lemma \ref{lem:If--commute,}, $\tilde{f}$ is in $\widehat{S_{n}[l^{\infty}]}[l]$,
and by lemma \ref{lem:(1)-A-function} it extends to an element of
$\widehat{S_{n}}[l]$. Thus,

\[
\dim_{F_{l}}\widehat{S_{n}}[l]\geq\left(\begin{array}{c}
n\\
l^{e}
\end{array}\right)\frac{(l^{e}-1)!}{l^{e}-l^{e-1}}.
\]

It remains to prove that the right hand side is not less than the
multiplicity of $l$ in $n!$. Indeed, by a theorem due to Legendre,

\[
val_{l}(n!)=\sum_{i\geq1}\lfloor\frac{n}{l^{i}}\rfloor<\frac{n}{l-1}.
\]

If $n\neq l^{e}$ then $\dim_{F_{l}}\widehat{S_{n}}[l]\geq n\geq\frac{n}{l-1}>val_{l}(n!)$.

If $n=l^{e}$ and $n\geq6$, then using the inequality $(n-1)!\geq n^{2}$
we get $\dim_{F_{l}}\widehat{S_{n}}[l]\geq\frac{(n-1)!}{n(1-\frac{1}{l})}\geq\frac{n}{l-1}$.

For $l^{e}=n\leq5$, $\frac{(n-1)!}{n(1-\frac{1}{l})}\geq val_{l}(n!)$
holds as well.\end{proof}
\begin{rem}
On cycles of smaller $l$-power order, we are not always free to define
the image of a semicharacter. Consider for example $2$-cycles. Any
semicharacter takes the values $\pm1$ on any $2$-cycle. The identity
$(12)(34)\cdot(13)(24)=(14)(23)$ holds in the Klein $4$-group which
is commutative. Hence for any semicharacter $f$ of $S_{n}$, we have
$\prod_{1\leq i<j\leq4}f((ij))=1$. Moreover, we can replace the indices
$1,2,3,4$ with $t_{1}<t_{2}<t_{3}<t_{4}$, and get a linear relation
modulo $2$ between $f((t_{i}t_{j})),1\leq i<j\leq4$. Thus, the numbers
$f((ij))$ for $1\leq i<j\leq n$ satisfy a system of $\binom{n}{4}$
equations in $\binom{n}{2}$ variables. For $n=7$ it can be checked
that the only solutions are $f((ij))\equiv1$ and $f((ij))\equiv-1$.
Thus, the same holds for any $n\geq7$. This also shows that a character
on an abelian subgroup $H$ of a group $G$ may not always be extended
to a semicharacter of $G$.
\end{rem}
Let us now consider the alternating groups. The embedding $A_{n}\rightarrow S_{n}$
induces a natural restriction map $R:\widehat{S_{n}}\rightarrow\widehat{A_{n}}$.
Obviously, the sign map $sgn:S_{n}\rightarrow\{\pm1\}$ is in the
kernel of this map.
\begin{lem}
\label{thm:Ker R}(1) $\ker R\subseteq\widehat{S_{n}}[2]$

(2) If $n>1$ is not of the form $2^{k}$or $2^{k}+1$, then $\ker R=\{1,sgn\}$.\end{lem}
\begin{proof}
(1) If $\phi\in\ker R$, then for $g\in S_{n}$, $\phi(g)^{2}=\phi(g^{2})=1$.

(2) By the assumption, $n\geq6$. Let $\phi\in\ker R$, then for any
$i\neq j$ we have $\phi((ij))=\pm1$. Moreover, for distinct $i,j,k,l$
we have $\phi((ij))\phi((kl))=\phi((ij)(kl))=1,$ hence $\phi((ij))=\phi((kl))$.
From here it is easy to see that $\phi$ is constant on all the involutions
of the form $(ij)$. Multiplying if necessary by $sgn$, we may assume
that $\phi((ij))=1$ for all $i\neq j$. For every $r$ such that
$2^{r}\leq n$, and any $2^{r}-cycle$ $c\in S_{n}$, we may find,
by the assumption on $n$, $i$ and $j$ such that $(ij)$ is disjoint
from $c$, hence $\phi(c(ij))=1\Rightarrow\phi(c)=1$. Hence, $\phi$
is equal to $1$ on $S_{n}[2^{\infty}]$, and by lemma \ref{lem:(1)-A-function},
$\phi=1$.\end{proof}
\begin{thm}
\label{thm:An}Conjecture \ref{conjecture} holds for $A_{n}$. \end{thm}
\begin{proof}
For a prime $l>2$, we have $A_{n}[l^{\infty}]=S_{n}[l^{\infty}],$
hence by lemma \ref{lem:(1)-A-function},
\[
|\widehat{A_{n}}[l^{\infty}]|=|\widehat{A_{n}[l^{\infty}]}|=|\widehat{S_{n}[l^{\infty}]}|=|\widehat{S_{n}}[l^{\infty}]|.
\]
By theorem \ref{thm:Sn} we conclude that $val_{l}(|\widehat{A_{n}}|)=val_{l}(|\widehat{S_{n}}|)\geq val_{l}(|S_{n}|)=val_{l}(|A_{n}|).$
Thus, it remains to prove that $val_{2}(|\widehat{A_{n}}|)\geq val_{2}(|A_{n}|).$
We consider the following cases:
\begin{itemize}
\item If $n$ is not of the form $2^{k}$ or $2^{k}+1$, by lemma \ref{thm:Ker R}
we have $|\ker R|=2$, hence $val_{2}(|\widehat{A_{n}}|)=val_{2}(|\widehat{S_{n}}|)-1\geq val_{2}(|S_{n}|)-1=val_{2}(|A_{n}|)$. 
\item If $n\leq3$, $A_{n}$ is abelian.
\item If $n=4$, $A_{n}[2^{\infty}]$ is an abelian group of size 4 (the
Klein $4$-group), hence $\widehat{A_{n}[2^{\infty}]}$ has $4$ elements.
\item If $n=5$, we take the natural embedding $A_{4}\hookrightarrow A_{5}$
(where the permutations in the image fix the number $5$). Then, one
may extend the elements of $\widehat{A_{4}[2^{\infty}]}$ constructed
above to $\widehat{A_{5}[2^{\infty}]}$ by setting the functions to
be $1$ outside of $\widehat{A_{4}[2^{\infty}]}$. This shows that
$val_{2}(|\widehat{A_{5}}|)\geq2=val_{2}(|A_{5}|)$.
\item If $n=2^{k}+\epsilon$, where $\epsilon\in\{0,1\}$ and $k\geq3$,
let $m=2^{k}$. We will construct elements of $\widehat{A_{n}}[2]$
by a small variation on the proof of Theorem \ref{thm:Sn}: Let $A$
be the set of $m/4$-cycles in $S_{n}$, with the equivalence relation
defined by $\pi\sim\pi^{i}$ for $\pi\in A$ and odd $i$, and let
$V$ be the $\mathbb{F}_{2}$-vector space of functions $f:A\rightarrow\mathbb{F}_{2}^{+}$
satisfying $f(\pi^{i})=i\cdot f(\pi)$ for odd $i$. We have $dim_{\mathbb{F}_{2}}(V)=|A/\sim|.$
For any $f\in V$, we extend $f$ to a function $\tilde{f}:A_{n}[2^{\infty}]\rightarrow\mathbb{F}_{2}^{+}$
by $\tilde{f}(\pi)=\sum_{c\in c(\pi^{2},m/4)}f(c)$, where $c(-,-)$
is as in the proof of Theorem \ref{thm:Sn}. We claim that $\tilde{f}\in\widehat{A_{n}[2^{\infty}]}[2]$.
Indeed, let us assume that $\pi_{1},\pi_{2}\in A_{n}[2^{\infty}]$
commute, and let $\pi_{3}=\pi_{1}\pi_{2}$. Since $\pi_{i}\in A_{n}$,
no $\pi_{i}$ is an $m$-cycle. If no $\pi_{i}$ contains an $m/2$-cycle,
then no $\pi_{i}^{2}$ contains an $m/4$-cycle, hence $\tilde{f}(\pi_{i})=0$
for all $i$, and $\tilde{f}(\pi_{3})=\tilde{f}(\pi_{1})+\tilde{f}(\pi_{2})$.
If $\pi_{1}$ contains an $m/2$-cycle $\sigma_{1}$, then $\pi_{2}$
(acting by conjugation) either fixes $\sigma_{1}$ or moves it to
an additional $m/2$-cycle of $\pi_{1}$. Hence, $\pi_{2}^{2}$ fixes
$\sigma_{1}$. Consequently, there exists $k$ such that $\pi_{2}^{2}$
is equal to $\sigma_{1}^{k}$ on the support of $\sigma_{1}$. Since
$\pi_{2}$ is not an $m$-cycle, $k$ must be even. We conclude that
$\pi_{2}^{2}$ is a power of $\pi_{1}^{2}$ on the support of any
$m/4$-cycle of $\pi_{1}^{2}$, and similarly that $\pi_{i}^{2}$
is a power of $\pi_{j}^{2}$ on the support of any $m/4$-cycle of
$\pi_{j}$, for all $i,j\in\{1,2,3\}$. Hence, there exist disjoint
cycles $c_{1},\dotsc,c_{u}$ and integers $a_{ij}$ such that $\pi_{i}|_{supp(c_{j})}=c_{j}^{a_{ij}}$,
and any $m/4$-cycle of any $\pi_{i}$ is supported on $supp(c_{j})$
for some $j$. From here, it follows that $\tilde{f}(\pi_{3})=\tilde{f}(\pi_{1})+\tilde{f}(\pi_{2})$.
We conclude, using lemma \ref{lem:(1)-A-function}, that $\dim_{F_{2}}\widehat{A_{n}}[2]\geq\dim_{\mathbb{F}_{2}}V=\binom{n}{m/4}\frac{(m/4-1)!}{(m/4-m/8)}\geq n\geq val_{2}(|A_{n}|)$.
\end{itemize}
\end{proof}

\section{Unitriangular Groups}

For groups of upper unitriangular matrices (i.e unipotent upper triangular
matrices), we will construct semicharacters by a discrete version
of the $\log$ function.
\begin{defn}
For all $n$ and prime $p$, let $w_{n,p}$ be the polynomial 
\[
w_{n,p}(x)=\sum_{i=1}^{n-1}(-1)^{i+1}\frac{p^{e}}{i}x^{i}\in\mathbb{Z}_{(p)}[x]
\]

where $e$ is the number satisfying $p^{e}\leq n-1<p^{e+1}$.
\end{defn}
Note that modulo $p$, $w_{n,p}(x)\equiv\sum_{ip^{e}<n}(-1)^{i+1}\frac{x^{ip^{e}}}{i}$. 
\begin{lem}
$Let$ $p$ be prime, \label{logmult} and let $A$ be a nilpotent
$\mathbb{F}_{p}$-algebra of nilpotence degree $n$.
\begin{enumerate}
\item If $x,y\in A$ commute then 
\[
w_{n,p}(x+y+xy)=w_{n,p}(x)+w_{n,p}(y).
\]

\item If $p\geq n$ then $w_{n,p}$ defines a bijection from $A$ to itself.
\end{enumerate}
\end{lem}
\begin{proof}
$\,$
\begin{enumerate}
\item We have the formal identity over $\mathbb{Q}$, 
\begin{eqnarray*}
\sum_{i\geq1}(-1)^{i+1}\frac{(x+y+xy)^{i}}{i} & = & \log((1+x)(1+y))=\log(1+x)+\log(1+y)\\
 & = & \sum_{i\geq1}(-1)^{i+1}\left(\frac{x^{i}}{i}+\frac{y^{i}}{i}\right)
\end{eqnarray*}
We multiply it by $p^{e}$, and take the resulting identity modulo
monomials of degree $\geq n$. If $I$ is the ideal in $\mathbb{Q}[x,y]$
generated by these monomials, we get an identity in $\mathbb{Q}[x,y]/I$,
\[
w_{n,p}(xy+x+y)=w_{n,p}(x)+w_{n,p}(y).
\]
Since all the coefficients are in $\mathbb{Z}_{(p)}$, and $A$ is
an $\mathbb{F}_{p}$-algebra, the identity holds for any two commuting
elements in $A$.
\item Since $p\geq n$, we have $e=0$. let $u(x)=\sum_{i=1}^{n-1}\frac{x^{i}}{i!}$.
By the formal identity $exp(log(1+x))-1=x$, we get, by the same argument,
that $u(w(x))=x$ for all $x\in A$.
\end{enumerate}
\end{proof}
\begin{thm}
\label{thm:unipotent}Let $\mathbb{F}$ be a finite field, let $n\leq char(\mathbb{F}),$
and let $G=U(n,\mathbb{F})$ be the group of upper triangular unipotent
matrices. Then, conjecture \ref{conjecture} holds for $G$.\end{thm}
\begin{proof}
Let $N$ be the additive group of nilpotent upper triangular matrices
in $M_{n}(\mathbb{F})$. We define the map $\log:G\rightarrow N$
by $\log(A)=w_{n,p}(A-I)$ ,where $p$ is the characteristic of $\mathbb{F}$.
By lemma \ref{logmult}, $\log$ is a bijection, and if $A,B$ commute
then $\log(AB)=\log(A)+\log(B)$. Thus, the map $\widehat{N}\rightarrow\widehat{G}$
given by $\phi\mapsto\phi\circ\log$ is a monomorphism, hence $|\widehat{G}|$
is divisible by $|N|=|G|$.
\end{proof}

\section{General Linear Groups}

In this section, we consider the case of $G=GL(n,q)$ where $q$ is
a power of a prime $p$. In the case $n=2$, we will prove that conjecture
\ref{conjecture} holds for $G$. As in previous cases, the proof
is carried out locally at each prime divisor $l$ of $|G|$. Naturally,
the cases $l=p$ and $l\neq p$ are completely different from one
another. At $l=p$ our proof is valid for all $n$.

\subsection{Sylow subgroups of $GL(n,q)$}

Let us first review the Sylow subgroups of $GL(n,q)$. These have
been completely described in the literature (see \cite{Weir}, \cite{Carter&Fong},
\cite{Alali}). We will follow the approach of \cite{Alali} which
is most suited to our needs. Due to a few misprints in \cite{Alali},
we reproduce the results here.
\begin{lem}
\label{lem:Syl for l|q-1}Assume that $l|q-1$, and let $L$ be an
$l$-Sylow subgroup of $G$.
\begin{enumerate}
\item If $l>2$ or $q\equiv1\mod4$ then $L\cong H_{1}\wr H_{2}$ where
$H_{1}=\mathbb{F}_{q}^{*}[l^{\infty}]$ (i.e. the unique $l$-Sylow
subgroup of $F_{q}^{*}$) and $H_{2}$ is an $l$-Sylow subgroup of
$S_{n}$. In particular, $L$ may be embedded in the subgroup of $G$
of monomial matrices (that is, matrices that have exactly one nonzero
entry in each row).
\item If $l=2$ and $q\equiv3\mod4$, let $H$ be the subgroup of $End_{\mathbb{F}_{q}}(\mathbb{F}_{q^{2}})$
generated by the elements of $\mathbb{F}_{q^{2}}^{*}[2^{\infty}]$
(acting by multiplication in \textup{$\mathbb{F}_{q^{2}}$), and the
Frobenius automorphism $x\mapsto x^{q}$ of $\mathbb{F}_{q^{2}}$.
By choosing a basis for $\mathbb{F}_{q^{2}}/\mathbb{F}_{q}$, we view
$H$ as a subgroup of $GL_{2}(\mathbb{F}_{q})$. Then,}

\begin{enumerate}
\item If $n$ is even then $L\cong H\wr H_{2}$ where $H_{2}$ is a $2$-Sylow
subgroup of $S_{n/2}$. In particular, $L$ may be embedded in the
subgroup of $G$ of $2\times2$-block matrices, that are monomial
with respect to these blocks, and such that each block is in $H$.
\item If $n$ is odd then $L\cong\{\pm1\}\times L'$ where $L'$ is an $l$-Sylow
subgroup of $GL_{n-1}(\mathbb{F}_{q})$.
\end{enumerate}
\end{enumerate}
\end{lem}
\begin{proof}
~
\begin{enumerate}
\item By \cite[lemma 2.3 (i)]{Alali}, we have%
\footnote{Note that the formula in \cite{Alali} should read $|GL(n,q)|_{r}=(q^{d}-1)_{r}^{[\frac{n}{d}]}([\frac{n}{d}])!$
and not as stated.%
}
\[
val_{l}(|G|)=n\cdot val_{l}(q-1)+val_{l}(n!).
\]
Hence, $|H_{1}\wr H_{2}|=l^{val_{l}(|G|)}$. Since $H_{1}\wr H_{2}$
may be embedded in $G$ as a group of monomial matrices, $L\cong H_{1}\wr H_{2}$.
\item Let us first assume that $n=2$. We have $val_{2}(|G|)=val_{2}((q^{2}-1)(q-1))=val_{2}(|F_{q^{2}}^{*}|)+1.$
On the other hand, since $H\cong\mathbb{F}_{q^{2}}^{*}[2^{\infty}]\rtimes Gal(\mathbb{F}_{q^{2}}/\mathbb{F}_{q})$,
we have $val_{2}(|H|)=val_{2}(|F_{q^{2}}^{*}|)+1$, hence $L\cong H$.
More generally, if $n$ is even, then by \cite[lemma 2.3 (ii)]{Alali},
$val_{2}(|G|)=(n/2)(val_{2}(|F_{q^{2}}^{*}|)+1)+val_{2}((n/2)!)=val_{2}(H\wr H_{2})$,
so that $L\cong H\wr H_{2}$. If $n$ is odd, then $val_{2}|GL(n,q)|=val_{2}(q^{n}-1)+val_{2}(GL(n-1,q))=1+val_{2}(GL(n-1,q))$.
Thus the group $\left\{ \left(\begin{array}{cc}
\pm1 & 0\\
0 & g
\end{array}\right)\,|\, g\in L'\right\} \cong\{\pm1\}\times L'$ is a $2$-Sylow subgroup of $G$. 
\end{enumerate}
\end{proof}
\begin{lem}
\label{lem:FieldEmbedding}If $l>2$, $l\neq p$ and $d=min\{i:\mbox{ }l|q^{i}-1\},$
then an $l$-Sylow subgroup of $G$ is isomorphic to an $l$-Sylow
subgroup of $GL(\left\lfloor \frac{n}{d}\right\rfloor ,q^{d})$.
\end{lem}
This is proved in \cite[remark 2.5]{Alali}.%
\footnote{Note that remark 2.5 cited above is only valid if restricted to $l>2$
or $q\equiv1\mod4$. Otherwise, the claim that if $l|q-1$ then an
$l$-Sylow subgroup of $GL(n,q)$ may be embedded in the subgroup
of monomial matrices may not be valid. For example, in $GL(2,3)$
there are only $8$ monomial matrices while the 2-Sylow subgroups
are of order $16$.%
}

Note that $GL(\left\lfloor \frac{n}{d}\right\rfloor ,q^{d})$ may
be embedded in $GL(d\left\lfloor \frac{n}{d}\right\rfloor ,q)$ via
restriction of scalars, and $GL(d\left\lfloor \frac{n}{d}\right\rfloor ,q)$
may be embedded in $GL(n,q)$ by $g\mapsto\left(\begin{array}{cc}
g & 0\\
0 & I_{m\times m}
\end{array}\right)$ where $m=n-d\left\lfloor \frac{n}{d}\right\rfloor $. Finally,
\begin{lem}
\label{lem:p-sylow}A $p$-Sylow subgroup of $GL(n,q)$ is isomorphic
to the subgroup of unitriangular matrices $U(n,q)$.\end{lem}
\begin{proof}
We have $|U(n,q)|=q^{1+2+..+(n-1)}$ which is the largest power of
$p$ dividing the order of $GL(n,q)$.
\end{proof}

\subsection{The case $l\neq p$}

We let $G=GL(2,q)$. We have $|G|=(q^{2}-1)(q^{2}-q)$. We will now
prove that $val_{l}(|\widehat{G}|)\geq val_{l}(G)$ for all prime
divisors $l$ of $|G|$ other than $p$.
\begin{lem}
\label{lem:matrices}Let $k>2$ be a divisor of $q+1$. Then $G$
has exactly $\frac{q(q-1)}{2}$ cyclic subgroups of order $k$.\end{lem}
\begin{proof}
Let $\alpha\in\mathbb{F}_{q^{2}}$ be a primitive $k$-th root of
unity. Note that $\alpha\notin\mathbb{F}_{q}$. Let $m(t)=t^{2}-At+B\in\mathbb{F}_{q}[t]$
be the minimal polynomial of $\alpha$. We have $m|t^{k}-1$ and $m\nmid t^{k'}-1$
for $1\leq k'<k$. For any matrix $g\in G$, if $\textnormal{tr}(g)=A$
and $\det(g)=B$ then $m(g)=0$, hence $g^{k}=1$. Also, $g$ is not
scalar (otherwise, there would exist a root of $m$ in $\mathbb{F}_{q}$),
so $g^{k'}\neq1$ for $1\leq k'<k$ (otherwise, $g$ would be a root
of $gcd(t^{k'}-1,m)=1$). We conclude that $ord(g)=k$. Let us count
the matrices $g$ that satisfy the above conditions: There are $q$
pairs of elements of $a,d\in\mathbb{F}_{q}$ that sum to $A$. For
any such pair, we have $ad\neq B$ (otherwise, $a$ and $d$ would
be roots of $m$). Hence there are $(q-1)$ pairs $(b,c)$ such that
$ad-bc=B$. We conclude that there are $q(q-1)$ matrices in $G$
whose characteristic polynomial is $m$, and they all have order $k$.
On the other hand, if $g\in G$ is a matrix of order $k$, let $m$
be the characteristic polynomial of $g$. Since $g$ is not scalar,
$m$ is also the minimal polynomial of $g$. We have $m|x^{k}-1$,
and $m\nmid x^{k'}-1$ for $1\leq k'<k$. Hence, the roots of $m$
are primitive $k$-th roots of unity. We conclude that $m$ has two
distinct roots $\alpha_{1},\alpha_{2}\in\mathbb{F}_{q^{2}}\setminus\mathbb{F}_{q}$,
and $m$ is their minimal polynomial. We have established a $q(q-1):1$
correspondence between elements of $g$ of order $k$ and unordered
pairs $\{\alpha_{1},\alpha_{2}\}$ of $Gal(\mathbb{F}_{q^{2}}/\mathbb{F}_{q})$-conjugate
primitive $k$-th roots of unity in $\mathbb{F}_{q^{2}}$. Therefore,
there are $\frac{q(q-1)}{2}\varphi(k)$ elements of order $k$ in
$G$, and correspondingly, $\frac{q(q-1)}{2}$ cyclic subgroups of
order $k$.\end{proof}
\begin{lem}
\label{lem:l divides q-1=00003D>conj}Let $l\neq2$ be a divisor of
$q-1$, then $val_{l}(|\widehat{G}|)\geq val_{l}(|G|)$.\end{lem}
\begin{proof}
Since an $l$-Sylow subgroup of $S_{2}$ is trivial, by lemma \ref{lem:Syl for l|q-1}
the group of diagonal matrices whose diagonal entries are in $\mathbb{F}_{q}^{*}[l^{\infty}]$
is an $l$-Sylow subgroup of $G$. Let $A$ be the set of $1$-dimensional
subspaces of $\mathbb{F}_{q}^{2}$, and let $B={A \choose 2}$, so
that $|B|=\frac{q(q+1)}{2}$. For any $\{V_{1},V_{2}\}\in B$, let
$D_{V_{1},V_{2}}$ be the group of matrices in $G$ which have $V_{1}$
and $V_{2}$ as invariant subspaces, with eigenvalues in $\mathbb{F}_{q}^{*}[l^{\infty}]$.
These groups are the $l$-Sylow subgroups of $G$. Let $C$ be the
set of functions which are defined on $B$ and assign to any $\{V_{1},V_{2}\}$
in $B$ an element of $\hom(D_{V_{1},V_{2}},\mathbb{C}^{*})$ which
is trivial on scalar matrices. $C$ is an abelian group of size $l^{val_{l}(q-1)\cdot\frac{q(q+1)}{2}}$.
For any $\phi\in C$ let us construct an element $T_{\phi}\in\widehat{G[l^{\infty}]}$:
Given $g\in G[l^{\infty}]$, $g$ is in some $l$-Sylow subgroup,
hence $g\in D_{V_{1},V_{2}}$ for some $\{V_{1},V_{2}\}\in B$ (note
that unless $g$ is scalar, the set $\{V_{1},V_{2}\}$ is uniquely
determined by $g$). We define $T_{\phi}(g)=\phi(\{V_{1},V_{2}\})(g)$.
If $g_{1},g_{2}\in G[l^{\infty}]$ commute, then either one of them
is scalar, or they have the same eigenvectors. In both cases, we have
$T_{\phi}(g_{1}g_{2})=T_{\phi}(g_{1})T_{\phi}(g_{2})$. Hence, $T_{\phi}\in\widehat{G[l^{\infty}]}$.
Using lemma \ref{lem:(1)-A-function} we get a monomorphism from $C$
to $\widehat{G}[l^{\infty}]$. We conclude that $val_{l}(|\widehat{G}|)\geq val_{l}(q-1)\cdot\frac{q(q+1)}{2}\geq2val_{l}(q-1)=val_{l}(|G|)$.\end{proof}
\begin{lem}
\label{lem:l divides q+1=00003D>conj}Let $l\neq2$ be a divisor of
$q+1$, then $val_{l}(|\widehat{G}|)\geq val_{l}(|G|)$.\end{lem}
\begin{proof}
Let $r=val_{l}(q+1)=val_{l}(q^{2}-1)$. By lemma \ref{lem:FieldEmbedding},
the $l$-Sylow subgroups of $G$ are isomorphic to $\mathbb{F}_{q^{2}}^{*}[l^{\infty}]$.
Hence they are cyclic of order $l^{r}$. By lemma \ref{lem:matrices}
there are $\frac{q(q-1)}{2}$ distinct $l$-Sylow subgroups in $G$.
By lemma \ref{lem:CyclicSylow}, $val_{l}|\widehat{G}|\geq\frac{q(q-1)}{2}\geq r=val_{l}(|G|)$
(the inequality $\frac{q(q-1)}{2}\geq r$ can be verified directly
for $q=2$, and for $q>2$ we have $\frac{q(q-1)}{2}\geq q\geq r$).\end{proof}
\begin{lem}
If $p\neq2$ then $val_{2}(|\widehat{G}|)\geq val_{2}(|G|)$.\end{lem}
\begin{proof}
Let us distinguish two cases:
\begin{itemize}
\item If $q\equiv1\mod4$ then let $r=val_{l}(q-1)$. By lemma \ref{lem:Syl for l|q-1},
the subgroup $L\leq G$ of monomial matrices whose nonzero entries
are in $\mathbb{F}_{q}^{*}[2^{\infty}]$ is a $2$-Sylow subgroup
of $G$. Note that for $g\in L$, $g^{2}$ is diagonal. Moreover,
if $g$ itself is diagonal, then the entries of $g^{2}$ are squares
of elements in $\mathbb{F}_{q}^{*}[2^{\infty}]$. Otherwise, $g^{2}$
is a scalar matrix. Let $S=\{x^{2}|x\in\mathbb{F}_{q}^{*}[2^{\infty}]\}$
(a cyclic group of size $2^{r-1})$, and let $I_{2}=\{cI|c\in\mathbb{F}_{q}^{*}[2^{\infty}]\}.$
For any $g\in G[2^{\infty}]$, $g^{2}$ is either in $I_{2}$ or diagonalizable
with eigenvalues in $S$. Using the same method as in lemma \ref{lem:l divides q-1=00003D>conj},
we shall construct elements of $\widehat{G[2^{\infty}]}$: Let $A$
and $B$ be as in the proof of lemma \ref{lem:l divides q-1=00003D>conj}.
For any $\{V_{1},V_{2}\}\in B$, let $D_{V_{1},V_{2}}$ be the group
of matrices in $G$ which preserve $V_{1}$ and $V_{2}$, with eigenvalues
in $S$. Let $C$ be the set of functions which are defined on $B$
and assign to any $\{V_{1},V_{2}\}\in B$ an element of $\hom(D_{V_{1},V_{2}},\mathbb{C}^{*})$
which is trivial on scalar matrices. $C$ is an abelian group of size
$2^{(r-1)\cdot\frac{q(q+1)}{2}}$. For any $\phi\in C$, let us construct
an element $T_{\phi}\in\widehat{G[2^{\infty}]}$: Given $g\in G[2^{\infty}]$,
$g^{2}$ is diagonalizable. If $g^{2}$ is scalar, we define $T_{\phi}(g)=1$.
Otherwise, let $\{V_{1},V_{2}\}\in B$ be such that $g^{2}\in D_{V_{1},V_{2}}$.
We define $T_{\phi}(g)=\phi(\{V_{1},V_{2}\})(g^{2})$. If $g,h\in G[2^{\infty}]$
commute then $g^{2}$ and $h^{2}$ commute as well. There is a conjugate
of $L$ that contains both $g$ and $h$ (and consequently $gh$),
hence there is $\{V_{1},V_{2}\}\in B$ such that $g^{2},h^{2},g^{2}h^{2}\in D_{V_{1},V_{2}}\cup I_{2}$.
Since $D_{V_{1},V_{2}}$ and $I_{2}$ are groups, either $g^{2},h^{2},g^{2}h^{2}\in D_{V_{1},V_{2}}$
or $g^{2},h^{2},g^{2}h^{2}\in I_{2}$. In both cases we have $T_{\phi}(gh)=T_{\phi}(g)T_{\phi}(h)$.
Therefore, $T_{\phi}\in\widehat{G[2^{\infty}]}.$ Using lemma \ref{lem:(1)-A-function}
we get a monomorphism from $C$ to $\widehat{G}[2^{\infty}]$. We
conclude that $val_{2}(|\widehat{G}|)\geq(r-1)\cdot\frac{q(q+1)}{2}$.
By the assumption on $q$, we have $q\geq5$. Since $r\geq2$, we
have $val_{2}(|\widehat{G}|)\geq15(r-1)\geq2r+1=val_{2}(|G|)$. 
\item If $q\equiv3\mod4$ then let $r=val_{2}(q^{2}-1)$. By lemma \ref{lem:Syl for l|q-1},
the $2$-Sylow subgroups of $G$ are isomorphic to $\mathbb{F}_{q^{2}}^{*}[2^{\infty}]\rtimes Gal(\mathbb{F}_{q^{2}}/\mathbb{F}_{q})$,
thus are dihedral groups of order $2^{r+1}$. Our proof will be a
variation on the proof of lemma \ref{lem:l divides q+1=00003D>conj}.
Let $L_{1},\dotsc,L_{N}$ be the $2$-Sylow subgroups of $G$. For
each $i$, there is a unique cyclic subgroup $C_{i}\leq L_{i}$ of
order $2^{r}$. Given $N$ homomorphisms $\phi_{i}:C_{i}\rightarrow\mathbb{F}_{2}^{+}$
($1\leq i\leq N$), we have for any $1\leq i<j\leq N$, that both
$\phi_{i}$ and $\phi_{j}$ are identically $0$ on $C_{i}\cap C_{j}$.
Hence, the $N$-tuple $(\phi_{1},\dotsc,\phi_{N})$ glues to a function
from $\cup_{i}C_{i}$ to $\mathbb{F}_{2}^{+}$. Let us extend this
function to a function $\phi:G[2^{\infty}]\rightarrow\mathbb{F}_{2}^{+}$
by defining $\phi(g)=0$ for $g\notin\cup C_{i}$. We claim that $\phi\in\widehat{G[2^{\infty}]}$.
Indeed, if $g,h\in G[2^{\infty}]$ commute, they must lie in the same
Sylow subgroup $L_{i}$. If both $g$ and $h$ are not of order $2^{r}$,
then so is $gh$ and we have $\phi(gh)=0=\phi(g)+\phi(h)$. Otherwise,
we may assume without loss of generality that $g$ is of order $2^{r}$.
Then $g$ generates $C_{i}$, hence $h\in C_{i}$ (otherwise $L_{i}$
would be abelian), hence $\phi(gh)=\phi(g)+\phi(h)$. Using lemma
\ref{lem:(1)-A-function} we get a monomorphism from $\prod\hom(C_{i},\mathbb{F}_{2}^{+})$
to $\widehat{G}[2^{\infty}]$. We conclude that $|\widehat{G}|$ is
divisible by $2^{N}$. Finally, $N$ is bounded from below by the
number of cyclic subgroups of $G$ of size $2^{r}$. By lemma \ref{lem:matrices},
we have $N\geq\frac{q(q-1)}{2}$. On the other hand, $val_{2}|G|=val_{2}(q+1)+2\leq q+2$.
For $q=3$ the claim can be verified directly (see section \ref{sec:Numerical-evidence}).
If $q>3$ then by the assumption on $q$, $q\geq7$, and we have $val_{2}(|\widehat{G}|)\geq N\geq\frac{q(q-1)}{2}\geq3q\geq q+2\geq val_{2}(|G|)$.
\end{itemize}
\end{proof}

\subsection{The case $l=p$}
\begin{lem}
$val_{p}(|\widehat{G}|)\geq2val_{p}(|G|).$\end{lem}
\begin{proof}
Let $N$ be the algebra of upper triangular nilpotent $n$ by $n$
matrices, and let $U=I+N$ be the unipotent group. By lemma \ref{lem:p-sylow},
$U$ is a $p$-Sylow subgroup of $G$. We define a function $F:G[p^{\infty}]\rightarrow M(n,q)$
by $F(g)=w_{n,p}(g-1)$, where the polynomial $w_{n,p}$ was defined
before lemma \ref{logmult}. Note that $F(hgh^{-1})=hF(g)h^{-1}$.
If $x,y\in G[p^{\infty}]$ commute, then they lie in a common Sylow
subgroup, hence we may assume that they lie in $U$. Let $a=x-I,b=y-I$,
then $a,b\in N$. By lemma \ref{logmult},

$F(xy)=F((I+a)(I+b))=w_{n,p}(ab+a+b)=w_{n,p}(a)+w_{n,p}(b)=F(I+a)+F(I+b))=F(x)+F(y).$

We will denote by $F(g)_{ij}$ the $(i,j)$ entry of $F(g).$ Given
any $n^{2}-n$ linear functionals $\left(\phi_{ij}\in\hom(\mathbb{F}_{q},\mathbb{F}_{p})\right)_{1\leq i\neq j\leq n}$,
the function $T_{(\phi_{ij})}:G[p^{\infty}]\rightarrow\mathbb{F}_{p}^{+}$
defined by $T_{(\phi_{ij})}(g)=\sum_{i\neq j}\phi_{ij}(F(g)_{ij})$
is in $\widehat{G[p^{\infty}]}[p]$, and by lemma \ref{lem:(1)-A-function}
extends to an element of $\widehat{G}[p]$. Let us see that in this
way, we get a subspace of $\widehat{G}[p]$ of dimension $dim(\mathbb{F}_{q}/\mathbb{F}_{p})(n^{2}-n)=2val_{p}(|G|)$:
Indeed, any matrix of the form $aE_{ij}$ (where $E_{ij}$ is the
elementary matrix with only one nonzero entry $1$ at $(i,j)$) for
$i\neq j$ is in the image of $F$, because if $d$ is the least degree
in the polynomial $w_{n,p}\, mod\, p$, and $M$ is the nilpotent
matrix
\[
\begin{pmatrix}0 & a\\
 & 0 & 1 &  &  & 0\\
 &  & \ddots & \ddots\\
 &  &  & 0 & 1\\
 &  &  &  & 0 & 0\\
 &  & 0 &  &  & \ddots & \ddots\\
 &  &  &  &  &  & 0 & 0\\
 &  &  &  &  &  &  & 0
\end{pmatrix}
\]
 with $d-1$ $1's$, then $w_{n,p}(M)=M^{d}$ has only one nonzero
entry equal to $a$, and by conjugating with a permutation matrix
we get a matrix $A$ satisfying $F(I+A)=w_{n,p}(A)=aE_{ij}$. Hence,
if $T_{(\phi_{ij})}\equiv0$ then $\phi_{ij}=0$ for all $i\neq j$.
\end{proof}
We conclude:
\begin{thm}
\label{thm:GL(n)}For all $q$, conjecture \ref{conjecture} holds
for $GL(2,q).$ 
\end{thm}

\section{Numerical evidence\label{sec:Numerical-evidence}}

When a finite group $G=\{g_{1},\dotsc,g_{n}\}$ is given via its multiplication
table, the group $\widehat{G}$ can be calculated effectively in the
following way: Let $e_{1},\dotsc,e_{n}$ be the standard basis of
$\mathbb{Z}^{n}$. For each pair of commuting elements $g_{i},g_{j}$
we form the element $x(i,j)=e_{i}+e_{j}-e_{k}$, where $g_{i}g_{j}=g_{k}$.
Then $\widehat{G}\cong\mathbb{Z}^{n}/<x(i,j)|g_{i}g_{j}=g_{j}g_{i}>$.
Thus $|\widehat{G}|$ (and also the isomorphism type of $\widehat{G}$)
can be calculated in polynomial time in $|G|$.

We have applied this algorithm to the list of isomorphism types of
groups of order $\leq255$, using the computer algebra system \texttt{GAP}
and the \texttt{SmallGroups} library (there are $7012$ isomorphism
types of groups of order $\leq255$). It turns out that conjecture
\ref{conjecture} holds for all these groups. Since there are $56092$
non-isomorphic groups of order $256$, our computation power was not
sufficient to check the conjecture for all groups of that size.

$\,$

\address{Department of Mathematics and Computer Science, The Open University
of Israel, 1 University Road, p.o. box 808, Raanana, Israel.}

\email{E-mail address: gilal@openu.ac.il}

\begin{thebibliography}{References}
\bibitem[1]{Alali} I.M. AlAli, C. Hering and A. Neumann, \emph{A
number theoretic approach to Sylow $r$-subgroups of classical groups.}
(English summary) Rev. Mat. Complut. 18 (2005), no. 2, 329\textendash{}338. 

\bibitem[2]{Carter&Fong}R. Carter and P. Fong, \emph{The Sylow 2-subgroups
of the fi{}nite classical groups}, J. Algebra 1 (1964), 139\textendash{}151.

\bibitem[3]{Kimmerle} W. Kimmerle, R.Lyons, R. Sandling, and D.N.
Teague, \emph{Composition factors from the group ring and Artin's
Theorem on orders of simple groups}. Proc. London Math. Soc. (3) 60
(1990), no. 1, 89\textendash{}122.

\bibitem[4]{Weir}A. J. Weir, \emph{Sylow p-subgroups of the classical
groups over fi{}nite fi{}elds with characteristic prime to p}, Proc.
Amer. Math. Soc. 6 (1955), 529\textendash{}533.

\end{thebibliography}
\end{document}